\documentclass[12pt,reqno]{amsart}
\usepackage{amsmath,amssymb,amsthm}
\usepackage{geometry}
\usepackage{orcidlink}

\title{On an Overpartition Analogue of \(SOME(n)\) }

\author{D. S. Gireesh$^{1}$\orcidlink{0000-0002-2804-6479} and {B. Hemanthkumar$^{2}$\orcidlink{0000-0001-7904-293X}}}
\address{$^{1}$Department of Mathematics, B.M.S. College of Engineering, P.O. Box No.: 1908, Bull Temple Road,
Bengaluru-560 019, Karnataka, India.}
\email{gireeshdap@gmail.com}

\address{$^2$Department of Mathematics, RV College of Engineering, RV Vidyanikethan Post, Mysore Road, Bengaluru-560 059, Karnataka, India.}
\email{hemanthkumarb.30@gmail.com}


\newtheorem{theorem}{Theorem}

\newtheorem{corollary}[theorem]{Corollary}
\newtheorem{lemma}[theorem]{lemma}
\begin{document}

\begin{abstract}
Recently, Andrews and Dastidar introduced the partition function $SOME(n)$, defined as the sum of all the odd parts in the partitions of $n$ minus the sum of all the even parts in the partitions of $n$. They derived its generating function and established some congruences satisfied by \(SOME(n)\). In this paper, we introduce an overpartition analogue of $SOME(n)$, denoted by $\overline{SOME}(n)$, the sum of all the odd parts in the overpartitions of \(n\) minus the sum of all the even parts in the overpartitions of \(n\). We derive the generating function for $\overline{SOME}(n)$ and obtain congruences modulo \(3, \ 5\) and powers of \(2\). Our method is based on classical $q$-series identities and manipulations of infinite products and sums.
 \end{abstract}

\maketitle

\section{Introduction}
A partition of a nonnegative integer $n$ is a finite sequence
\(
\lambda = (\lambda_1,\lambda_2,\ldots,\lambda_k)
\)
such that
\(
\lambda_1 \ge \lambda_2 \ge \cdots \ge \lambda_k \ge 1\)
and
\(\sum_{i=1}^{k} \lambda_i = n.
\)
We write $\lambda \vdash n$ to denote that $\lambda$ is a partition of \(n\). Let $p(n)$ denote the number of such partitions of $n$. 
By convention, $p(0)=1$. One of the famous results in partition theory, due to Ramanujan, is the congruence
\[
p(5n+4) \equiv 0 \pmod{5},
\]
which states that the number of partitions of integers of the form \(5n+4\) is always divisible by \(5\) . 
In their recent paper, Andrews and Dastidar\cite{AD} introduced a new weighted partition function, denoted by $SOME(n)$. For a given \(n\), the function is defined as the total of all odd parts minus the total of all even parts taken over every partition of $n$. They obtained the generating function 
\[
\sum_{n=1}^\infty SOME(n) q^n
=
\frac{1}{(q;q)_\infty}
\sum_{m=1}^\infty
\frac{q^{m}}{(1+q^{m})^2}
\]
and proved that this function satisfies a similar congruence
\[SOME(5n+4) \equiv 0 \pmod{5}.\]
 Here and throughout this paper, we use the standard $q$-Pochhammer symbol notation
\[
(a;q)_\infty = \prod_{n=0}^{\infty} (1-aq^n).
\]

An overpartition of a nonnegative integer $n$ is a partition of $n$ 
in which the first occurrence of each distinct part may be overlined. 
The theory of overpartitions was first systematically studied by Corteel and Lovejoy\cite{CL}. We denote by $\overline{p}(n)$ the number of overpartitions of $n$. By convention, $\overline{p}(0)=1$. The generating function for $\overline{p}(n)$ is given by
\[
\sum_{n=0}^{\infty} \overline{p}(n) q^n 
= \frac{(-q;q)_\infty}{(q;q)_\infty}.
\]

Several partition identities extend naturally to overpartitions, often leading to richer combinatorial and arithmetic structures. Consequently, the arithmetic properties of overpartitions have attracted considerable attention in recent years. In particular, many authors have established numerous Ramanujan-type congruences for overpartition functions modulo powers of  $2, 3,$ and $5$  (see, for example,  \cite{CD, DL, JZ, T, Y}).

Motivated by the results of Andrews and Dastidar on $SOME(n)$, it is natural to ask whether similar weighted statistics defined on overpartitions satisfy analogous arithmetic properties. In this paper, we introduce an overpartition analogue of $SOME(n)$, denoted by $\overline{SOME}(n)$, defined as the sum of all odd parts minus the sum of all even parts taken over all overpartitions of $n$. That is,
\[
\overline{SOME}(n)
=
\sum_{\lambda \vdash n}
\left(
\text{sum of odd parts of } \lambda
-
\text{sum of even parts of } \lambda
\right),
\]
where the sum is taken over all overpartitions $\lambda$ of $n$.

We derive an explicit expression for the generating function of $\overline{SOME}(n)$ and establish several arithmetic properties, including nonnegativity and divisibility results.

In \cite{HS}, Hirschhorn and Sellers showed that
\begin{align*}
    \overline{p}(4n+3) &\equiv 0 \pmod{8},\\
    \overline{p}(8n+7) &\equiv 0 \pmod{64}.
\end{align*}
Here we prove that the same congruences hold for the function $\overline{SOME}(n)$. Consequently, the sum of odd parts in the overpartitions of $4n+3$ and $8n+7$ is divisible by $8$ and $64$, respectively, and the same holds for the sum of even parts.

Our main results are as follows:
\begin{theorem}\label{thm1}
We have
\begin{equation}\label{eq1}
\sum_{n=1}^\infty \overline{SOME}(n) q^n
=
2\,\frac{(-q;q)_\infty}{(q;q)_\infty}
\sum_{m=1}^\infty
\frac{q^{2m-1}}{(1+q^{2m-1})^2}.    
\end{equation}
\end{theorem}
\begin{corollary}
For all $n \ge 1$, $\overline{SOME}(n)$ is even.
\end{corollary}
This is an immediate consequence of the preceding theorem.
\begin{corollary}\label{thm12}
For all $n \ge 0$,
\(
\overline{SOME}(n) \ge 0.
\)
\end{corollary}
\begin{theorem}\label{thm2}
For all $n \ge 0$,
\[
\overline{SOME}(n) = 2\sum_{k=0}^n \overline{p}(k)\sigma_{o-e}(n-k),
\]
where $\sigma_{o-e}(a)$ denotes the sum of the odd divisors $d$ of $a$ for which $a/d$ is odd, minus the sum of the even divisors $d$ of $a$ for which $a/d$ is odd.
\end{theorem}
\begin{corollary}
    For all $n \ge 1$,
\[
\sum_{k=0}^n \overline{p}(k)\sigma_{o-e}(n-k)\ge 0.
\]
\end{corollary}
This follows directly from Theorem \ref{thm2} and Corollary \ref{thm12}. 
\begin{theorem}\label{thm3}
For each $n \ge 1$, we have
\begin{equation}\label{eq24}
\overline{S}_{b}(n) = 2b\sum_{k=0}^{\lfloor{\frac{n-b}{2b}}\rfloor} \overline{p}(n-(2k+1)b),
\end{equation}
where $\overline{S}_{b}(n)$ denotes the sum of all parts equal to $b$ appearing in all overpartitions of n.
\end{theorem}
The following corollary follows directly from the above theorem.
\begin{corollary}
    For each $n \ge 1$, we have
\[
\overline{S}_{b}(n)\equiv 0 \pmod{2b}.
\]
\end{corollary}
\begin{theorem}\label{thm4}
    For each $n\geq 0$, we have
    \begin{align}
    \overline{SOME}(4n+3)&\equiv 0\pmod{8}, \label{eq2} \\
        \overline{SOME}(8n+7)&\equiv 0\pmod{64} \label{eq3}.
    \end{align}
\end{theorem}
Let \(\overline{S_o}(n)\) (respectively \(\overline{S_e}(n)\)) denote the sum of the odd (respectively even) parts in the overpartitions of \(n\). 
\begin{corollary}\label{thm41}
    For each $n\geq 0$, we have
    \begin{align}
        \overline{S_o}(4n+3) &\equiv \overline{S_e}(4n+3) \equiv 0 \pmod{8}, \label{eq02} \\
        \overline{S_o}(8n+7) &\equiv \overline{S_e}(8n+7) \equiv 0 \pmod{64}.\label{eq03}
    \end{align}
\end{corollary}
\begin{theorem} \label{thm5}
For each $n\geq 0$, we have
    \begin{align}
        \overline{SOME}(16n)&\equiv 2^2 \overline{SOME}(4n) \pmod{2^7}, \label{eq4}\\
        \overline{SOME}(32n)&\equiv 2^2 \overline{SOME}(8n) \pmod{2^9}.\label{eq5}
    \end{align}
\end{theorem}
\begin{theorem}\label{thm6}
    For each $n\ge 0$ and $k\in\{1,2,3,4\}$, we have
    \begin{equation}\label{eq6}
        \overline{SOME}(2^{2k-1}n)\equiv 0\pmod{2^{2k+1}}.
    \end{equation}
\end{theorem}
\begin{theorem}\label{thm7}
   If $n$ is not a perfect square, then for each $k\in\{1,2,3,4\}$, 
    \begin{equation}\label{eq7}
        \overline{SOME}(4^{k-1}n)\equiv 0\pmod{2^{2k+1}}.
    \end{equation}
\end{theorem}
The following corollary is an immediate consequence of the above theorem.
\begin{corollary}
    For any prime $p$, if $r$ is a quadratic nonresidue modulo $p$, then for each $k\in\{1,2,3,4\}$,
\begin{equation}\label{eq8}
\overline{SOME}(4^{k-1}(pn+r)) \equiv 0 \pmod{2^{2k+1}}.
\end{equation}
\end{corollary}
\begin{theorem} \label{thm8}
    For each $n\geq 0$, we have
    \begin{align}
        \overline{SOME}(4n+2)&\equiv 0\pmod{2^5}, \label{eq9}\\
        \overline{SOME}(8n+2)&\equiv 0\pmod{2^7}, \label{eq10}\\
        \overline{SOME}(16n+10)&\equiv 0\pmod{2^8}, \label{eq11}\\
        \overline{SOME}(16n+14)&\equiv 0\pmod{2^6}, \label{eq12}\\
         \overline{SOME}(16n+12)&\equiv 0\pmod{2^5}, \label{eq13}\\
        \overline{SOME}(32n+28)&\equiv 0\pmod{2^7}, \label{eq14}\\
        \overline{SOME}(32n+20)&\equiv 0\pmod{2^5}, \label{eq15}\\
        \overline{SOME}(32n+24)&\equiv 0\pmod{2^6}, \label{eq16}\\
         \overline{SOME}(128n+80)&\equiv 0\pmod{2^7}, \label{eq17}\\
         \overline{SOME}(64n+48)&\equiv 0\pmod{2^7}, \label{eq18}\\
        \overline{SOME}(128n+112)&\equiv 0\pmod{2^8},\label{eq19}\\
        \overline{SOME}(128n+96)&\equiv 0\pmod{2^8},\label{eq191}\\
        \overline{SOME}(256n+192)&\equiv 0\pmod{2^9},\label{eq192}\\
        \overline{SOME}(512n+320)&\equiv 0\pmod{2^9}.\label{eq193}
    \end{align}
\end{theorem}
\begin{theorem} \label{thm10}
    For each $n\ge 0$, we have
    \begin{align}
        \overline{SOME}(3n+2)&\equiv 0\pmod{3},\label{eq22}\\
        \overline{SOME}(29n+19)&\equiv 0\pmod{3}.\label{eq23}
    \end{align}
\end{theorem}
\begin{theorem} \label{thm9}
    For each $n\ge 0$, we have
    \begin{align}
       \overline{SOME}(40n+31)&\equiv 0\pmod{5},\label{eq20}\\
        \overline{SOME}(40n+39)&\equiv 0\pmod{5}.\label{eq21}
    \end{align}
\end{theorem}

Section \ref{sec2} presents the necessary background material required for the proofs of our results. The proofs of Theorems \ref{thm1}–-\ref{thm3} are given in Section \ref{sec3}. In Section \ref{sec4}, we establish Theorem \ref{thm4} along with its corollary. The proofs of Theorems \ref{thm5}–-\ref{thm8} are provided in Section \ref{sec5}. Finally, Section \ref{sec6} contains the proofs of Theorems \ref{thm9} and \ref{thm10}.

\section{Preliminaries}\label{sec2}
In this section, we collect several preliminary results that will be used throughout the paper. 
We begin with Ramanujan's general theta function $f(a,b)$, defined by
\[
f(a,b) = \sum_{n=-\infty}^{\infty} a^{n(n+1)/2} b^{n(n-1)/2}, \ \ |ab|<1.
\]

Jacobi’s triple product identity takes the form
\[
f(a,b) = (-a;ab)_\infty\, (-b;ab)_\infty\, (ab;ab)_\infty.
\]

The special cases of $f(a,b)$ are

\begin{align}
\psi(q) &:= f(q,q^3) 
= \sum_{n=0}^{\infty} q^{n(n+1)/2}
= \frac{(q^2;q^2)_\infty}{(q;q^2)_\infty}, \label{eq25}\\
\phi(q) &:= f(q,q) 
= 1+2\sum_{n=1}^{\infty} q^{n^2}=\dfrac{(-q;q^2)_\infty (q^2;q^2)_\infty}{(q;q^2)_\infty (-q^2;q^2)_\infty},\label{eq26}
\end{align}
and
\begin{equation}
f(-q) := f(-q,-q^2)
= \sum_{n=-\infty}^{\infty} (-1)^n q^{n(3n-1)/2}
= (q;q)_\infty.
\end{equation}

Consider the identities due to Ramanujan \cite[Entry 25, p. 40]{BCB}:
\begin{align}
\phi(-q^2)^2 &= \phi(q)\phi(-q),  \label{eq27}\\
\phi(q) &= \phi(q^4) + 2q\psi(q^8), \label{eq28}\\
\phi(q)^2 &= \phi(q^2)^2 + 4q\psi(q^4)^2. \label{eq29}
\end{align}

\begin{lemma} For $|q|<1$, the following holds:
    \begin{align}
        q\frac{d}{dq}\log{\phi(q)} &= 2\sum_{n=1}^{\infty}\frac{q^{2n-1}}{(1+q^{2n-1})^2}\label{Ep1} \\ 
        & = \frac{2\sum_{k=1}^\infty k^2 q^{k^2}}{1+2\sum_{k=1}^\infty  q^{k^2}}.\label{Ep2}
    \end{align}
\end{lemma}
\begin{proof}   
    From \eqref{eq26}, we see that
    \begin{align*}
        q\frac{d}{dq}\log\phi(q)&=q\frac{d}{dq}\left(\sum_{k=1}^\infty\left(\log(1+q^{2k-1})+\log(1-q^{2k})-\log(1-q^{2k-1})-\log(1+q^{2k})\right)\right)\\
&=q\frac{d}{dq}\left(\sum_{k=1}^\infty\sum_{n=1}^{\infty}\left(\frac{(q^{2k-1})^n}{n}\left((-1)^{n-1}+1\right)-\frac{(q^{2k})^n}{n}\left((-1)^{n-1}+1\right)\right)\right)\\&
=q\frac{d}{dq}\left(\sum_{k=1}^\infty\sum_{n=1}^{\infty}\left(\frac{2(q^{2k-1})^{2n-1}}{2n-1}-\frac{2(q^{2k})^{2n-1}}{2n-1}\right)\right)\\
&
=q\frac{d}{dq}\left(\sum_{n=1}^{\infty}\frac{2}{2n-1}\left(\sum_{k=1}^{\infty}(-1)^{k-1}(q^{2n-1})^k\right)\right)\\&
=2q\frac{d}{dq}\left(\sum_{n=1}^{\infty}\frac{q^{2n-1}}{(2n-1)(1+q^{2n-1})}\right)\\
& = 2\sum_{n=1}^{\infty}\frac{q^{2n-1}}{(1+q^{2n-1})^2}.
    \end{align*}
The second equality is derived in the same way, by applying logarithmic differentiation to the series expansion of $\phi(q)$.
\end{proof}

\section{Proofs of theorems \ref{thm1} to \ref{thm3}} \label{sec3}

\subsection{Proof of Theorem \ref{thm1}}
The steps are outlined as follows.
\[
\sum_{n\ge1} \overline{SOME}(n) q^n
=
\left.
\frac{\partial}{\partial z}
\prod_{n=1}^{\infty}
\frac{(1+(zq)^{2n-1})(1+(q/z)^{2n})}
{(1-(zq)^{2n-1})(1-(q/z)^{2n})}
\right|_{z=1}.
\]

Differentiating logarithmically and evaluating at $z=1$ gives
\begin{align*}
\nonumber\sum_{n\ge1} \overline{SOME}(n) q^n&=\frac{(-q;q)_\infty}{(q;q)_\infty}
\left[
\sum_{n=1}^{\infty}\frac{(2n-1)q^{2n-1}}{1+q^{2n-1}}-\sum_{n=1}^{\infty}\frac{(2n)q^{2n}}{1+q^{2n}}\right.\\&\left.
\nonumber+\sum_{n=1}^{\infty}
\frac{(2n-1)q^{2n-1}}{1-q^{2n-1}}-\sum_{n=1}^{\infty}
\frac{(2n)q^{2n}}{1-q^{2n}}
\right]\\&\nonumber
=\frac{(-q;q)_\infty}{(q;q)_\infty}
\sum_{n=1}^{\infty}
(-1)^{n-1} n q^n
\left(
\frac{1}{1+q^n}
+
\frac{1}{1-q^n}
\right)\\&
=2\,
\frac{(-q;q)_\infty}{(q;q)_\infty}
\sum_{n=1}^{\infty}
\frac{(-1)^{n-1} n q^n}{1-q^{2n}}.
\end{align*}

Expanding the denominator as a geometric series,
\[
\frac{1}{1-q^{2n}}
=
\sum_{m=0}^{\infty} q^{2mn},
\]
we obtain
\begin{align*}
\sum_{n\ge1} \overline{SOME}(n) q^n &= 2\, \frac{(-q;q)_\infty}{(q;q)_\infty}
\sum_{n, m\ge1}
(-1)^{n-1} n q^{(2m-1)n}\\
& = 2\,
\frac{(-q;q)_\infty}{(q;q)_\infty}
\sum_{m=1}^{\infty}
\frac{q^{2m-1}}{(1+q^{2m-1})^2}.
\end{align*}
This completes the proof.

In view of \eqref{eq1} and \eqref{Ep2}, we see that
\begin{align}\label{eq30}
\nonumber\sum_{n=1}^\infty \overline{SOME}(n) q^n
&
=\frac{1}{\phi(-q)}\times \frac{2\sum_{k=1}^{\infty} k^2 q^{k^2}}{1+2\sum_{k=1}^{\infty} q^{k^2}}\\&
=2\frac{1}{\phi(-q^2)^2}\sum_{k=1}^{\infty} k^2 q^{k^2}.  
\end{align}
This proves Corollary \ref{thm12}.

\subsection{Proof of Theorem \ref{thm2}}
By Theorem \ref{thm1}, we have
\begin{align*}
\sum_{n\ge1} \overline{SOME}(n) q^n
&=2\,\frac{(-q;q)_\infty}{(q;q)_\infty}
\sum_{m=1}^{\infty}
\frac{q^{2m-1}}{(1+q^{2m-1})^2}\\&
=2\,\frac{(-q;q)_\infty}{(q;q)_\infty}
\sum_{m=1}^{\infty}\sum_{n=1}^{\infty} (-1)^{n-1}nq^{n(2m-1)},
\end{align*}
which can be written as
\begin{align*}
\sum_{n\ge1} \overline{SOME}(n) q^n
&=2\sum_{m=0}^{\infty}\overline{p}(m)q^m
\sum_{k=1}^{\infty}
\sum_{\substack{d \mid k \\ \frac{k}{d} \text{ odd}}} (-1)^{d-1}dq^k
\\&
=2\sum_{m=0}^{\infty}\overline{p}(m)q^m
\sum_{k=1}^{\infty}
\sigma_{o-e}(k)q^k.
\end{align*}
The proof follows from the above identity.

\subsection{Proof of Theorem \ref{thm3}}
We proceed as follows:
\begin{align*}
\sum_{n\ge1}\overline{S}_{b}(n) q^n
=&
\left.
\frac{\partial}{\partial z}
\prod_{\substack{n=1\\n\neq b}}^{\infty}
\frac{(1+q^{n})(1+(zq)^{b})}
{(1-q^{n})(1-(zq)^{b})}
\right|_{z=1}\\&
=\frac{(-q;q)_\infty}{(q;q)_\infty}\left(\frac{bq^b}{1+q^b}+\frac{bq^b}{1-q^b}\right)\\&
=\frac{(-q;q)_\infty}{(q;q)_\infty}\left(\frac{2bq^b}{1-q^{2b}}\right)\\&
=2b\sum_{n=0}^{\infty}\overline{p}(n)q^n
\sum_{k=0}^{\infty}q^{(2k+1)b}.
\end{align*}
The result follows from the last equality.

\section{Proofs of Theorem \ref{thm4} and Corollary \ref{thm41}}\label{sec4}
From \eqref{eq30}, it is evident that 
\begin{equation*}
    \sum_{n=0}^\infty \overline{SOME}(2n+1)q^{2n+1} = 2\frac{1}{\phi(-q^2)^2}\sum_{m=0}^\infty (2m+1)^2 q^{(2m+1)^2}.
\end{equation*}
Replacing $q^2$ by $q$ and using \eqref{eq27}, we obtain
\begin{equation*}
    \sum_{n=0}^\infty \overline{SOME}(2n+1)q^{n} = 2\frac{\phi(q)^2}{\phi(-q^2)^4}\sum_{m=0}^\infty (2m+1)^2 q^{2m^2+2m}.
\end{equation*}
Invoking \eqref{eq29} in the above equation and then extracting the terms involving $q^{2n+1}$ from both sides,
\begin{equation}\label{pf2}
    \sum_{n=0}^\infty \overline{SOME}(4n+3)q^{n} = 8\frac{\psi(q^2)^2}{\phi(-q)^4}\sum_{m=0}^\infty (2m+1)^2 q^{m^2+m}.
\end{equation}
Next employing \eqref{eq27} and \eqref{eq29} in \eqref{pf2} and extracting the terms involving $q^{2n+1}$ from both sides,
\begin{equation}\label{pf3}
    \sum_{n=0}^\infty \overline{SOME}(8n+7)q^{n} = 64\frac{\psi(q)^6}{\phi(-q)^8}\sum_{m=0}^\infty (2m+1)^2 q^{m(m+1)/2}.
\end{equation}
Congruence \eqref{eq2} and \eqref{eq3} follow from \eqref{pf2} and \eqref{pf3}, respectively.
Clearly,
\begin{align*}
    \overline{S_o}(4n+3) + \overline{S_e}(4n+3) &= (4n+3) \ \overline{p}(4n+3) \equiv 0 \pmod{8}, \\
    \overline{S_o}(8n+7) + \overline{S_e}(8n+7) &= (8n+7) \ \overline{p}(8n+7) \equiv 0 \pmod{64}
\end{align*}
and
\begin{align*}
    \overline{S_o}(4n+3) - \overline{S_e}(4n+3) &= \overline{SOME}(4n+3) \equiv 0 \pmod{8}, \\
    \overline{S_o}(8n+7) - \overline{S_e}(8n+7) &= \overline{SOME}(8n+7) \equiv 0 \pmod{64}.
\end{align*}
By adding and subtracting, we establish the Corollary \ref{thm41}.

\section{Proofs of Theorems \ref{thm5} to \ref{thm8}} \label{sec5}
From \eqref{eq30}, it is easy to see that
\begin{equation*}
    \sum_{n=0}^\infty \overline{SOME}(2n)q^{2n} = 2\frac{1}{\phi(-q^2)^2}\sum_{k=1}^\infty (2k)^2 q^{(2k)^2}
\end{equation*}
  which yields,
\begin{equation}\label{pf21}
    \sum_{n=0}^\infty \overline{SOME}(2n)q^{n} = 8\frac{\phi(q)^2}{\phi(-q^2)^4}\sum_{k=1}^\infty k^2 q^{2k^2}.
\end{equation}  
Using \eqref{eq29} in \eqref{pf21} and extracting the terms involving $q^{2n+1}$ and  $q^{2n}$ from both sides, 
\begin{equation}\label{pf22}
    \sum_{n=0}^\infty \overline{SOME}(4n+2)q^{n} = 32\frac{\psi(q^2)^2}{\phi(-q)^4}\sum_{k=1}^\infty k^2 q^{k^2}
\end{equation}
and
\begin{align}\label{pf23}
  \nonumber  \sum_{n=0}^\infty \overline{SOME}(4n)q^{n} &= 8\frac{\phi(q)^2}{\phi(-q)^4}\sum_{k=1}^\infty k^2 q^{k^2}\\ 
  &= 8\frac{\phi(q)^6}{\phi(-q^2)^8}\sum_{k=1}^\infty k^2 q^{k^2}.
\end{align}
We have
\begin{align}\label{pf24}
\nonumber\phi(q)^6 & = (\phi(q^2)^2+4q\psi(q^4)^2)^3 \\
          & = \phi(q^2)^6 +2^6 q^3 \psi(q^4)^6 + 12q \phi(q^2)^4 \psi(q^4)^2 + 3\cdot 2^4 q^2 \phi(q^2)^2 \psi(q^4)^4
\end{align}
and 
\begin{equation}\label{pf25}
    \frac{\phi(q)^{2^k}}{\phi(-q)^{2^k}} \equiv 1\pmod{2^{k+2}}.
\end{equation}
Substituting \eqref{pf24} into \eqref{pf23}, extracting the terms involving $q^{2n+1}$ and $q^{2n}$, and then using \eqref{pf25}, we obtain
\begin{equation}\label{pf26}
\sum_{n=0}^\infty \overline{SOME}(8n+4)q^{n} \equiv 2^3 \frac{1}{\phi(q)^2} \sum_{k=0}^\infty (2k+1)^2 q^{2k^2+2k} \pmod{2^7}
\end{equation}
and 
\begin{align}\label{pf27}
\nonumber\sum_{n=0}^\infty \overline{SOME}(8n)q^{n} &\equiv 2^5 \frac{1}{\phi(q)^2} \sum_{k=1}^\infty k^2 q^{2k^2}\\
&+ 3\cdot2^5 \frac{\psi(q^2)^2}{\phi(-q)^4} \sum_{k=0}^\infty (2k+1)^2 q^{2k^2+2k+1} \pmod{2^9}.
\end{align}
Using \eqref{eq29} in \eqref{pf27} and then extracting the terms involving $q^{2n+1}$ and $q^{2n}$,
\begin{equation}\label{pf28}
\sum_{n=0}^\infty \overline{SOME}(16n+8)q^{n} \equiv 3 \cdot 2^5\frac{\psi(q)^2\phi(q)^4}{\phi(-q)^8} \sum_{k=0}^\infty (2k+1)^2 q^{k^2+k} \pmod{2^7}
\end{equation}
and 
\begin{equation}\label{pf29}
\sum_{n=0}^\infty \overline{SOME}(16n)q^{n} \equiv 2^5 \frac{1}{\phi(q)^2} \sum_{k=1}^\infty k^2 q^{k^2} + 3 \cdot 2^8 \psi(q)^6 \sum_{k=0}^\infty (2k+1)^2 q^{k^2+k+1} \pmod{2^9}.
\end{equation}
Using the fact that $\psi(q)^2 \equiv\psi(q^2) \pmod{2}$ in the above identity and extracting the terms involving the odd and even powers of $q$ from both sides of \eqref{pf29}, we obtain
\begin{equation}\label{pf30}
\sum_{n=0}^\infty \overline{SOME}(32n+16)q^{n} \equiv 2^5 \phi(q)^2 \sum_{k=0}^\infty (2k+1)^2 q^{2k^2+2k} \pmod{2^8}
\end{equation}
and
\begin{equation}\label{pf31}
   \sum_{n=0}^\infty \overline{SOME}(32n)q^{n} 
\equiv2^7 \sum_{k=1}^\infty k^2 q^{2k^2} - 2^7 \psi(q^2)^2 \sum_{k=0}^\infty (2k+1)^2 q^{2k^2+2k+1} \pmod{2^9}.
\end{equation}
It follows directly from \eqref{pf31} that
\begin{equation}\label{pf32}
   \sum_{n=0}^\infty \overline{SOME}(64n+32n)q^{n} \equiv 2^7 \psi(q^2) \sum_{k=0}^\infty (2k+1)^2 q^{k^2+k} \pmod{2^8}
\end{equation}
and 
\begin{equation}\label{pf33}
   \sum_{n=0}^\infty \overline{SOME}(64n)q^{n} \equiv 2^7 \sum_{k=1}^\infty k^2 q^{k^2} \pmod{2^9}.
\end{equation}
Congruence \eqref{eq4} follows from \eqref{pf23} and \eqref{pf29}, while \eqref{eq5} follows from \eqref{pf27} and \eqref{pf31}. Moreover, Theorem \ref{thm6} follows from \eqref{pf21}, \eqref{pf27}, \eqref{pf31}, and \eqref{pf33}.

In view of \eqref{eq30}, \eqref{pf23}, \eqref{pf29}, and \eqref{pf33}, we see that
\begin{align*}
   \sum_{n=0}^\infty \overline{SOME}(2n)q^{n} &\equiv 2 \sum_{k=1}^\infty k^2 q^{k^2} \pmod{2^3},\\
   \sum_{n=0}^\infty \overline{SOME}(4n)q^{n} &\equiv 2^3 \sum_{k=1}^\infty k^2 q^{k^2} \pmod{2^5},\\
   \sum_{n=0}^\infty \overline{SOME}(16n)q^{n} &\equiv 2^5 \sum_{k=1}^\infty k^2 q^{k^2} \pmod{2^7},\\
   \sum_{n=0}^\infty \overline{SOME}(64n)q^{n} &\equiv 2^7 \sum_{k=1}^\infty k^2 q^{k^2} \pmod{2^9}.
\end{align*}
This completes the proof of Theorem \ref{thm7}.

From \eqref{pf22}, we have
\begin{equation}\label{pf34}
    \sum_{n=0}^\infty \overline{SOME}(4n+2)q^{n} \equiv 2^5\psi(q^2)^2\sum_{k=1}^\infty k^2 q^{k^2} \pmod{2^8}
\end{equation}
which yields,
\begin{equation}\label{pf35}
    \sum_{n=0}^\infty \overline{SOME}(8n+2)q^{n} \equiv 2^7\psi(q^2)\sum_{k=1}^\infty k^2 q^{2k^2} \pmod{2^8}
\end{equation}
and
\begin{equation}\label{pf36}
    \sum_{n=0}^\infty \overline{SOME}(8n+6)q^{n} \equiv 2^5\psi(q^2)\sum_{k=0}^\infty (2k+1)^2 q^{2k^2+2k} \pmod{2^6}.
\end{equation}
Thus, congruence \eqref{eq9} -- \eqref{eq12} follow from \eqref{pf34} -- \eqref{pf36}.

Employing \eqref{eq29} in \eqref{pf26} and extracting the terms involving the odd and even powers of $q$ from both sides,
\begin{equation}\label{pf37}
\sum_{n=0}^\infty \overline{SOME}(16n+12)q^{n} \equiv -2^5 \psi(q^2)^2 \sum_{k=0}^\infty (2k+1)^2 q^{k^2+k} \pmod{2^7}
\end{equation}
and
\begin{equation}\label{pf38}
\sum_{n=0}^\infty \overline{SOME}(16n+4)q^{n} \equiv 2^3 \phi(q)^2 \sum_{k=0}^\infty (2k+1)^2 q^{k^2+k} \pmod{2^5}.
\end{equation}
Congruences \eqref{eq13} and \eqref{eq14} follow from \eqref{pf37}, while congruence \eqref{eq15} follows from \eqref{pf38}.

From \eqref{pf28}, we have
\begin{equation}\label{pf39}
\nonumber   \sum_{n=0}^\infty \overline{SOME}(16n+8)q^{n} \equiv 2^5\psi(q^2) \sum_{k=0}^\infty (2k+1)^2 q^{k^2+k} \pmod{2^6}.
\end{equation}
This establishes congruence \eqref{eq16}.

The following congruences follow directly from \eqref{pf30}:
\begin{equation}\label{pf40}
\sum_{n=0}^\infty \overline{SOME}(64n+16)q^{n} \equiv 2^5 \phi(q)^2 \sum_{k=0}^\infty (2k+1)^2 q^{k^2+k} \pmod{2^8}
\end{equation}
and 
\begin{equation}\label{pf41}
\sum_{n=0}^\infty \overline{SOME}(64n+48)q^{n} \equiv 2^7 \psi(q^2)^2 \sum_{k=0}^\infty (2k+1)^2 q^{k^2+k} \pmod{2^8}.
\end{equation}
Hence, congruence \eqref{eq17} follows from \eqref{pf40}; congruences \eqref{eq18} and \eqref{eq19} follow from \eqref{pf41}; and congruence \eqref{eq191} follows from \eqref{pf32}.

By \eqref{pf33},
\begin{equation}\label{pf42}
   \sum_{n=0}^\infty \overline{SOME}(128n+64)q^{n} \equiv 2^7 \sum_{k=1}^\infty (2k+1)^2 q^{2k^2+2k} \pmod{2^9}.
\end{equation}
This completes the proof of congruences \eqref{eq192} and \eqref{eq193}.
\section{Proofs of Theorems \ref{thm10} and \ref{thm9}} \label{sec6}
\noindent Recall \eqref{eq30},
\begin{equation*}
  \sum_{n=0}^\infty \overline{SOME}(n)q^n = 2\frac{\phi(-q^2)}{\phi(-q^2)^3}\sum_{k=1}^{\infty} k^2 q^{k^2}.
\end{equation*}
From \cite[p. 49]{BCB}, we have 
\[\phi(q) = \phi(q^9)+2qf(q^3, q^{15}).\]
Substituting this into the above identity, we obtain 
\begin{equation*}
  \sum_{n=0}^\infty \overline{SOME}(n)q^n \equiv 2\frac{(\phi(-q^{18})-2q^2f(-q^6, -q^{30}))}{\phi(-q^6)}\sum_{k=1}^{\infty} k^2 q^{k^2}.
\end{equation*}
Clearly $k^2 \equiv -1 \pmod{3}$ has no solution. Further, $k^2+2 \equiv 2 \pmod{3}$ only when $k\equiv 0 \pmod{3}$. Hence, coefficient of $q^{3n+2}$ vanishes modulo $3$. That is for $n\geq0$
\[\overline{SOME}(3n+2) \equiv 0 \pmod{3}.\]
Also, $k^2+2 \equiv 1 \pmod{3}$ implies $k^2 \equiv -1 \pmod{3}$. Thus, we have
\begin{equation*}
      \sum_{n=0}^\infty \overline{SOME}(3n+1)q^{3n+1} \equiv 2 \frac{\phi(-q^{18})}{\phi(-q^6)} \sum_{\substack{k=1 \\  k \not\equiv 0 \pmod{3}}}^\infty k^2q^{k^2}.
\end{equation*}

Replacing $q^3$ by $q$, we obtain
\begin{equation*}
\sum_{n=0}^{\infty} \overline{\text{SOME}}(3n+1)q^{n}
\equiv
2\frac{\phi(-q^{6})}{\phi(-q^{2})}
\sum_{\substack{k=1 \\ k\not\equiv0\pmod{3}}}^{\infty} k^2 q^\frac{k^2-1}{3}.
\end{equation*}
Extracting the terms involving the even powers of $q$ from both sides, we obtain
\begin{equation*}
    \sum_{n=0}^{\infty} \overline{\text{SOME}}(6n+1)q^{2n} 
\equiv 2\frac{\phi(-q^{6})}{\phi(-q^{2})}
\sum_{\substack{k=0 \\ k\not\equiv1\pmod3}}^{\infty} (2k+1)^2 q^\frac{(2k+1)^2-1}{3}.
\end{equation*}
Replacing $q^2$ by $q$, we obtain
\begin{align*}
    \sum_{n=0}^{\infty} \overline{\text{SOME}}(6n+1)q^{n} 
&\equiv 2\frac{\phi(-q^{3})}{\phi(-q)} 
\sum_{\substack{k=0 \\ k\not\equiv1\pmod3}}^{\infty} (2k+1)^2 q^\frac{2k^2+2k}{3} \pmod{3}\\
&\equiv 2\phi(-q)^2\sum_{\substack{k=0 \\ k\not\equiv1\pmod3}}^{\infty} (2k+1)^2 q^\frac{2k^2+2k}{3} \pmod{3}.
\end{align*}
Using \eqref{eq29} in the above identity and extracting terms involving $q^{2n+1}$, we obtain
\begin{equation*}
    \sum_{n=0}^{\infty} \overline{\text{SOME}}(12n+7)q^{n}  \equiv \psi(q^2)^2 \sum_{\substack{k=0 \\ k\not\equiv1\pmod3}}^{\infty} (2k+1)^2 q^\frac{k^2+k}{3}\pmod{3},
\end{equation*}
which yields
\[\overline{SOME}(24n+19) \equiv 0 \pmod{3}\]
proving Theorem \ref{thm10}.

Let us recall \eqref{pf3}.
\begin{equation*}
    \sum_{n=0}^\infty \overline{SOME}(8n+7)q^{n} = 64\frac{\psi(q)^6}{\phi(-q)^8}\sum_{m=0}^\infty (2m+1)^2 q^{m(m+1)/2}.
\end{equation*}
Since
\[\frac{\psi(q)^6}{\phi(-q)^8} = \frac{f(-q^2)^{20}}{f(-q)^{22}} \equiv f(-q)^3\frac{f(-q^{10})^4}{f(-q^5)^5} \pmod{5}\]
and 
\[f(-q)^3 = \sum_{k=0}^\infty (-1)^k (2k+1) q^{\frac{k(k+1)}{2}}. \]
We see that
\begin{equation*}
    \sum_{n=0}^\infty \overline{SOME}(8n+7)q^{n} \equiv 64\frac{f(-q^{10})^4}{f(-q^5)^5}\sum_{k=0}^\infty (-1)^k (2k+1) q^{\frac{k(k+1)}{2}}\sum_{m=0}^\infty (2m+1)^2 q^{\frac{m(m+1)}{2}} \pmod{5}.
\end{equation*}
We now determine when
\[\frac{k(k+1)}{2} + \frac{m(m+1)}{2} \equiv 4 \pmod{5}.\]
A straightforward calculation modulo $5$ shows that this occurs only when $m \equiv 2\pmod{5}$ and $k\equiv 1\pmod{5}$, or $m \equiv 1\pmod{5}$ and $k\equiv 2\pmod{5}$. In either case, we have $2m+1 \equiv 0 \pmod{5}$ or $2k+1 \equiv 0 \pmod{5}$. Therefore, the coefficient of $q^{5n+4}$ is divisible by $5$, and hence it vanishes modulo $5$. That is for $n\geq0$
\[\overline{SOME}(40n+39) \equiv 0\pmod{5}.\]
Also, we see that
\[\frac{k(k+1)}{2} + \frac{m(m+1)}{2} \equiv 3 \pmod{5}\]
occurs only when $m \equiv 2\pmod{5}$ and $k\equiv 0\pmod{5}$, or $m \equiv 0\pmod{5}$ and $k\equiv 2\pmod{5}$. Thus, either $2m+1 \equiv 0 \pmod{5}$ or $2k+1 \equiv 0 \pmod{5}$. Hence, the coefficient of $q^{5n+3}$ vanishes. Therefore, for $n\geq0$,
\[\overline{SOME}(40n+31) \equiv 0\pmod{5}.\]
This completes the proof of Theorem \ref{thm9}.

\subsection*{Funding} No funding available for this article.
\subsection*{Data Availability} Data sharing is not applicable to this article as no data sets were generated or analyzed.

\end{document}